\documentclass[10pt,reqno]{article}

\usepackage[b5paper,top=1.2in,left=0.9in]{geometry}
\usepackage{verbatim}
\usepackage{amssymb}
\usepackage{amsmath}
\usepackage{graphicx}
\usepackage{appendix}
\usepackage{color}
\usepackage{amsthm}
\usepackage{tikz}
\usetikzlibrary{arrows,positioning,decorations.pathmorphing,  decorations.markings} 

\newtheorem{Thm}{Theorem}[section]
\newtheorem{Lem}[Thm]{Lemma}
\newtheorem{Prop}[Thm]{Proposition}

\newtheorem{Rem}[Thm]{Remark}
\newtheorem{Def}[Thm]{Definition}

\newtheorem{Conj}[Thm]{Conjecture}

\newtheoremstyle{named}{}{}{\itshape}{}{\bfseries}{.}{.5em}{#1 #3}
\theoremstyle{named}

\def\R{\mathbb{R}}
\def\Q{\mathbb{Q}}

\def\Z{\mathbb{Z}}

\def\fb{\mathfrak{b}}
\def\g{\mathfrak{g}}
\def\fh{\mathfrak{h}}

\def\sl{\mathfrak{sl}}

\def\cA{\mathcal{A}}
\def\cB{\mathcal{B}}

\def\cD{\mathcal{D}}

\def\cH{\mathcal{H}}

\def\cN{\mathcal{N}}

\def\cP{\mathcal{P}}

\def\cU{\mathcal{U}}

\def\a{\alpha}
\def\b{\beta}
\def\c{\gamma}

\def\D{\Delta}

\def\d{\delta}
\def\e{\epsilon}
\def\ze{\zeta}

\def\l{\lambda}

\def\t{\tau}

\def\W{\Omega}
\def\w{\omega}

\def\bb{\textbf{b}}

\def\bc{\textbf{c}}

\def\be{\textbf{e}}

\def\bf{\textbf{f}}

\def\bH{\textbf{H}}

\def\bi{\textbf{i}}

\def\bo{\textbf{o}}

\def\bT{\textbf{T}}

\def\bU{\textbf{U}}

\def\=>{\Longrightarrow}

\def\corr{\longleftrightarrow}

\def\to{\longrightarrow}

\def\ox{\otimes}
\def\o+{\oplus}
\def\bo+{\bigoplus}

\def\<{\langle}
\def\>{\rangle}
\def\({\left(}
\def\){\right)}

\def\oo{\infty}

\def\^{\wedge}
\def\+{\dagger}

\def\inv{^{-1}}
\def\half{\frac{1}{2}}

\def\tab{\;\;\;\;\;\;}

\newcommand{\til}[1]{\widetilde{#1}}
\newcommand{\what}[1]{\widehat{#1}}

\renewcommand{\over}[1]{\overline{#1}}
\renewcommand{\vec}[1]{\overrightarrow{#1}}

\newcommand{\veca}[1]{\begin{pmatrix}#1 \end{pmatrix}}

\newcommand{\Eq}[1]{\begin{align}#1\end{align}}
\newcommand{\Eqn}[1]{\begin{align*}#1\end{align*}}

\begin{document}
\title{Positive representations, multiplier Hopf algebra, and continuous canonical basis}

\author{  Ivan C.H. Ip\footnote{
         	Kavli Institute for the Physics and Mathematics of the Universe (WPI), 
		The University of Tokyo, 
		Kashiwa, Chiba 
		277-8583, Japan
		\newline
		Email: ivan.ip@ipmu.jp
          }
}

\date{\today}

\numberwithin{equation}{section}

\maketitle

\begin{abstract}
We introduce the language of multiplier Hopf algebra in the context of positive representations of split real quantum groups, and discuss its applications with a continuous version of Lusztig-Kashiwara's canonical basis, which may provide a key to prove the closure of the positive representations under tensor products, and harmonic analysis of quantized algebra of functions in the sense of locally compact quantum groups.
\end{abstract}

\section{Introduction}\label{sec:intro}

The notion of \emph{positive principal series representations}, or \emph{positive representations} for short, was introduced by I. Frenkel and the author \cite{FI} as a new research program devoted to the representation theory of split real quantum groups $\cU_{q\til{q}}(\g_\R)$, where $q=e^{\pi ib^2}$ and its dual  $\til{q}=e^{\pi ib^{-2}}$. It uses the concept of modular double for quantum groups \cite{Fa1, Fa2}, and has been studied extensively for $U_{q\til{q}}(\sl(2,\R))$ by Teschner \textit{et al.} \cite{BT, PT1, PT2} in the context of Liouville conformal field theory. Explicit construction of the positive representations $\cP_\l$ of  $\cU_{q\til{q}}(\g_\R)$ associated to a simple Lie algebra $\g$ has been obtained for the simply-laced case in \cite{Ip2} and non-simply-laced case in \cite{Ip3}, where the generators $\{E_i, F_i, K_i\}$ of the quantum groups are realized by positive, unbounded, essentially self-adjoint operators acting on $L^2(\R^{l(w_0)})$. Here $l(w_0)$ is the length of the longest element $w_0$ in the Weyl group $W$. Furthermore, the so-called \emph{transcendental relations} of the (rescaled) generators:
\Eq{\til{\be_i}=\be_i^{\frac{1}{b_i^2}},\tab \til{\bf_i}=\bf_i^{\frac{1}{b_i^2}}, \tab \til{K_i}=K_i^{\frac{1}{b_i^2}}\label{trans}}
give the self-duality between different parts of the modular double, while in the non-simply-laced case, new explicit analytic relations between the quantum group $\cU_q(\g_\R)$ and its Langlands dual $\cU_{\til{q}}({}^L\g_\R)$ have been observed \cite{Ip3}.

It becomes clear in the work \cite{Ip1, Ip4} that the positivity of the generators is immediately connected with the notion of multiplier Hopf algebra introduced by van Daele \cite{VD}. Namely, by taking complex powers of the generators as unbounded operators, we can define a $C^*$-algebraic version of the Drinfeld-Jimbo quantum groups $\cU_{q\til{q}}^{C^*}(\g_\R)$ that is in some sense dual to certain locally compact quantum groups \cite{KV1,KV2}, based on earlier work by Baaj and Skandalis \cite{BS}, Woronowicz \cite{W} and others. This allows us to define the modular double as a quasi-triangular multiplier Hopf algebra with the universal $R$-operator being the canonical element of the corresponding multiplier Drinfeld's double, giving a braiding structure of the positive representations. 

More importantly, we believe that this establishes a link between the $C^*$-algebraic theory of Connes' non-commutative geometry and Woronowicz's matrix quantum groups, and the theory of Drinfeld-Jimbo type quantum groups at the algebraic and combinatoric level, providing both sides with new machineries. 

The construction of $\cU_{q\til{q}}^{C^*}(\g_\R)$ involves a continuous version of the PBW basis, and subsequently the concept of a continuous version of the \emph{canonical basis}, generalizing the finite dimensional concept first discovered by Lusztig \cite{Lu1,Lu2}, and subsequently proved to exist by Kashiwara \cite{Ka2}, using a different method under the name of \emph{global crystal basis}. One can think of it as an analogue to the relation between the (discrete) Fourier series in the harmonic analysis over the circle, which is compact, and the (continuous) Fourier transform in the harmonic analysis over the real line, which is non-compact, with the exponential functions in both cases being the corresponding basis. Furthermore, the use of the remarkable quantum dilogarithm function $G_b(x)$ and its variant $g_b(x)$ introduced earlier by Faddeev and Kashaev \cite{FKa}, will play the role of the $q$-factorial as well as quantum exponential function in the (continuous) split real setting. 

We believe that the nice properties of the canonical basis, especially its behavior under the decomposition of tensor products \cite{Ka2, Lu3}, can be generalized to the continuous version in certain settings, which enables us to complete the proof of the tensor category structure of the category $Rep^+(\cU_{q\til{q}}(\g_\R))$ of positive representations.

Moreover, the use of the multiplier Hopf algebra also allows us to study the dual of $\cU_{q\til{q}}^{C^*}(\g_\R)$, which appears to be the right object for the harmonic analysis of quantized algebra of functions, and coincides with the notion of matrix quantum groups in the locally compact quantum groups setting \cite{KV1, KV2}. With the machinery of the multiplicative unitary, the language of multiplier Hopf algebra provides insight in proving a version of Peter-Weyl theorem for the modular double of split real quantum groups.

As a final remark, the modular double $\cU_{q\til{q}}(\sl(2,\R))$ arising from Liouville theory is known to be closely related to quantum Teichm\"{u}ller theory through the construction of conformal blocks \cite{T}. On the other hand, it can also be constructed out of the representation theory of the Borel part of $\cU_{q\til{q}}(\sl(2,\R))$, which is the quantum plane \cite{FK}. With the help of the continuous canonical basis in the higher rank, the restriction of the positive representations of the modular double to the Borel part $\bU\bb\subset \cU_{q\til{q}}^{C^*}(\g_\R)$ is recently studied in \cite{Ip5}, which allows us to construct a quantized version of the higher Teichm\"{u}ller theory as introduced by Fock and Goncharov \cite{FG}, and connect with Toda field theories \cite{FL, Wy}.
\section{Multiplier Hopf Algebra}

In this section let us recall the basic definitions of multiplier Hopf algebra. For further details please refer to \cite{VD}. One of the original motivation of multiplier Hopf algebra is the following. Consider continuous functions $C[G]$ on a compact group $G$. Then the multiplication on $G$ induces a coproduct $\D:C[G]\to C[G]\ox C[G]$ given by
\Eq{\D f(g_1,g_2)=f(g_1g_2).}
When we discuss locally compact group $G$, due to the Gelfand-Naimark theorem, it is natural to consider the $C^*$-algebra $C_0(G)$ of functions vanishing at infinity. However, the formula for $\D$ above may no longer maps to $C_0(G)\ox C_0(G)$. Nonetheless, $\D f$ is still bounded and multiplication with elements in $C_0(G)\ox C_0(G)$ stays in the same space. Hence it is natural to extend the definition of the coproduct and introduce the following notion.

\begin{Def}
Let $\cB(\cH)$ be the algebra of bounded linear operators on a Hilbert space $\cH$. Then the multiplier algebra $M(\cA)$ of a $C^*$-algebra $\cA\subset \cB(\cH)$ is the $C^*$-algebra of operators
\Eq{M(\cA)=\{b\in \cB(\cH): b\cA\subset \cA, \cA b\subset \cA\}.}
In particular, $\cA$ is an ideal of $M(\cA)$.
\end{Def}
\begin{Def}\label{mHa}
A multiplier Hopf *-algebra is a $C^*$-algebra $\cA$ together with the antipode $S$, the counit $\e$, and the coproduct map
\Eq{\D:\cA\to M(\cA\ox \cA),}
all of which can be extended to a map from $M(\cA)$, such that the usual properties of a Hopf algebra holds on the level of $M(\cA)$.
\end{Def}

In particular, the standard example is $\cA=C_0(G)$, and $M(\cA)$ will be the algebra of bounded functions $C_b(G)$ on the locally compact group $G$.

Another motivation comes from the definition of the universal $R$-matrix, which is a very important solution to the Yang-Baxter equation in the theory of compact quantum groups giving the category of finite dimensional representation a braiding structure \cite{D}. However, $R$ is defined as an element of certain completion $\cU_\hbar(\g)\what{\ox}\cU_\hbar(\g)$ which is not very natural in the split real setting. In particular, for the theory of modular double $\cU_{q\til{q}}(\g_\R)$, the universal $R$ operator is defined by an integral transformation, and therefore the expression using generators of the split real quantum group does not lie in any algebraic completions. The following definition gives a natural interpretation of $R$ as an element of a multiplier Hopf algebra.

\begin{Def}\label{qtmHa}
A quasi-triangular multiplier Hopf algebra is a multiplier Hopf algebra $\cA$ together with an invertible element $R\in M(\cA\ox \cA)$ such that 
\Eq{(\D \ox id)(R)=&R_{13}R_{23}\in M(\cA\ox \cA\ox \cA),\\
(id\ox\D)(R)=&R_{13}R_{12} \in M(\cA\ox \cA\ox \cA),\\
\D^{op}(a)R=&R\D(a)\in M(\cA\ox \cA),\tab \forall a\in M(\cA),\\
(\e\ox id)(R)=&(id\ox \e)(R)=1\in M(\cA).
}
\end{Def}

In Section \ref{sec:CCB}, we will discuss the construction of a $C^*$-algebraic version of the Drinfeld-Jimbo type quantum group $\cU_{q\til{q}}(\g_\R)$ using a continuous version of the canonical basis, such that $R$ naturally arises as the canonical element of this multiplier algebra.

Finally in Section \ref{sec:harmonic}, we discuss the application of multiplier Hopf algebra in the context of harmonic analysis on the quantized algebra of functions on $G_{q\til{q}}^+(\R)$, which is certain dual algebra to $\cU_{q\til{q}}(\g_\R)$. Here the power of multiplier Hopf algebra comes with the tool $W$ called the \emph{multiplicative unitary}. More precisely, 

\begin{Def} Given a GNS-representation of the multiplier Hopf aglebra $\cA$ on the Hilbert space $\cH$, a \emph{multiplicative unitary} is a unitary operator $W\in \cB(\cH\ox \cH)$ satisfying the pentagon equation
\Eq{W_{23}W_{12}=W_{12}W_{13}W_{23},}
where the leg notation is used. 
\end{Def}

Then the following key properties are needed:
\begin{Prop} The multplicative unitary encodes the information of the coproduct, i.e.
\Eq{W^*(1\ox x)W = \D(x),\tab x\in\cA}
as operators on $\cH\ox \cH$.
\end{Prop}
\begin{Prop} The multiplicative unitary encodes the information of the Hopf dual $\what{\cA}$, i.e.
\Eq{\what{\cA}:=\{(\w\ox 1)W:\w\in\cB(\cH)^*\}^{\mbox{norm closure}}\subset \cB(\cH),}
and we have $W\in M(\cA\ox \what{\cA})$.
\end{Prop}
\begin{Prop} The multiplicative unitary gives the (left) regular co-representation of $\cA$ by
\Eq{\Pi: \cH&\to \cH\ox M(\cA)\\
f &\mapsto W'(f\ox 1),}
where $W'$ is $W_{21}$ viewed as an element in $\cB(\cH)\ox \cA$. By taking the dual we also obtain the (left) regular representation for the dual algebra.
\end{Prop}
We refer to \cite{Ip1, Tim} for more detailed discussions.

\section{Positive Representations of $\cU_{q\til{q}}(\g_\R)$}
The central object in the representation theory of split real quantum groups is the notion of positive representations \cite{FI}. In this section we will review its definitions and properties. For simplicity of this exposition, let us only consider simple Lie algebra $\g$ of simply-laced type \cite{Ip2}. Details for the non-simply-laced case can be found in \cite{Ip3}.

Let $q=e^{\pi ib^2}$ with $0<b<1$, $b^2\in \R\setminus\Q$, and consider the usual Drinfeld-Jimbo type quantum group $\cU_q(\g)$ generated by $\{E_i, F_i, K_i\}$ for $i\in I$ the root indices, satisfying:
\Eq{
K_iE_j&=q^{a_{ij}}E_jK_i,\\
K_iF_j&=q^{-a_{ij}}F_jK_i,\\
{[E_i,F_j]} &= \d_{ij}\frac{K_i-K_i\inv}{q-q\inv},
}
together with the Serre relations for $i\neq j$:
\Eq{
E_i^2E_j-(q+q\inv)E_iE_jE_i+E_jE_i^2&=0,\\
F_i^2F_j-(q+q\inv)F_iF_jF_i+F_jF_i^2&=0,
}
where $A:=(a_{ij})$ is the Cartan matrix.

\begin{Thm}\cite{FI, Ip2} Let 
\Eq{\be_i:=2\sin(\pi b^2)E_i,\tab \bf_i:=2\sin(\pi b^2)F_i.\label{smallef}}
Note that $2\sin(\pi b^2)=\left(\frac{\bi}{q-q\inv}\right)\inv>0$.
Then there exists a representation $\cP_{\l}$ of $\cU_{q}(\g)$ parametrized by the $\R_+$-span of the cone of positive weights $\l\in P_\R^+$, or equivalently by $\l\in \R_+^n$ where $n=rank(\g)$, such that 
\begin{itemize}
\item The generators $\be_i,\bf_i,K_i$ are represented by positive essentially self-adjoint operators acting on $L^2(\R^{l(w_0)})$, where $l(w_0)$ is the length of the longest element $w_0\in W$ of the Weyl group.
\item Define the transcendental generators:
\Eq{\til{\be_i}:=\be_i^{\frac{1}{b^2}},\tab \til{\bf_i}:=\bf_i^{\frac{1}{b^2}},\tab \til{K_i}:=K_i^{\frac{1}{b^2}}.\label{transdef}}
Then the representations of the generators $\til{\be_i},\til{\bf_i},\til{K_i}$ are obtained by replacing $b$ with $b\inv$ in the representations of the generators $\be_i,\bf_i,K_i$, and they generate the other part of the modular double $\cU_{\til{q}}(\g_\R)$ with $\til{q}=e^{\pi ib^{-2}}$.
\item The generators $\be_i,\bf_i,K_i$ and $\til{\be_i},\til{\bf_i},\til{K_i}$ commute weakly up to a sign.
\end{itemize}
\end{Thm}

The Theorem implies that $\cP_\l$ is simultaneously a representation for the split real form $\cU_q(\g_\R)$ and $\cU_{\til{q}}(\g_\R)$ by positive operators. Hence we will call $\cP_\l$  the positive representation of the modular double $$\cU_{q\til{q}}(\g_\R):=\cU_q(\g_\R)\ox \cU_{\til{q}}(\g_\R).$$
For the non-simply-laced case, the transcendental relations give rise to a direct analytic relation between the quantum group and its Langlands dual \cite{Ip3}.

Let us present the expressions of $\cP_\l$ in the case of type $A_1$ and $A_2$. By abuse of notation, let us denote by
\Eq{[u]e(p) := e^{\pi b (-u+2p)}+e^{\pi b(u+2p)},}
which is an unbounded positive operator. Here $p:=\frac{1}{2\pi i}\frac{\partial}{\partial u}$ is the standard momentum operator.

\begin{Prop}\cite{BT,PT2}\label{canonicalsl2} The positive representation $P_\l$ of $\cU_{q\til{q}}(\sl(2,\R))$ with ${\l\in\R_+}$ acting on $f(u)\in L^2(\R)$ is given by
\Eqn{
\be=&[u-\l]e(-p)=e^{\pi b(-u+\l-2p)}+e^{\pi b(u-\l-2p)},\\
\bf=&[-u-\l]e(p)=e^{\pi b(u+\l+2p)}+e^{\pi b(-u-\l+2p)},\\
K=&e^{-2\pi bu}.
}
\end{Prop}
\begin{Prop}\cite{Ip2}\label{typeA2} The positive representation $\cP_\l$ of $\cU_{q\til{q}}(\sl(3,\R))$ with parameters $\l:=(\l_1,\l_2)\in\R_+^2$, corresponding to the reduced expression $w_0=s_2s_1s_2$, acting on $f(u,v,w)\in L^2(\R^3)$, is given by
\Eqn{
\be_1=&[v-w]e(-p_v)+[u]e(-p_v+p_w-p_u),\\
\be_2=&[w]e(-p_w),\\
\bf_1=&[-v+u-2\l_1]e(p_v),\\
\bf_2=&[-2u+v-w-2\l_2]e(p_w)+[-u-2\l_2]e(p_u),\\
K_1=&e^{-\pi b(-u+2v-w+2\l_1)},\\
K_2=&e^{-\pi b(2u-v+2w+2\l_2)}.
}
\end{Prop}

\section{Continuous Canonical Basis}\label{sec:CCB}
Again for simplicity, in the following let us consider only the simply-laced case, further discussions can be found in \cite{Ip4}. Recall that in the theory of canonical basis introduced by Lusztig \cite{Lu1,Lu2}, the necessary ingredients are the rescaled simple root generators 
\Eq{\label{simple} E_i^{(N)}:=\frac{E_i^N}{[N]_q!},}
where $[n]_q=\frac{q^n-q^{-n}}{q-q\inv}$ is the quantum number, and the non-simple roots vectors given by Lusztig's isomorphism $T_i$:
\Eq{\label{nonsimple} E_{\a_k}^{(c)}:=T_{i_1}...T_{i_{k-1}}(E_{i_k}^{(c)}),}
such that the canonical basis is given by certain linear combinations of the PBW-basis of the form
\Eq{\label{cbasis}E_\bi^\bc := E_{\a_N}^{(c_N)}...E_{\a_2}^{(c_2)}E_{\a_1}^{(c_1)},}
where $w_0=s_{i_1}...s_{i_N}$ is an expression of the longest element of the Weyl group $W$.

In the context of positive representations, the generators $\{E_i\}_{i\in I}$ of $\cU_{q\til{q}}(\g_\R)$ are represented on $\cP_\l$ by positive self-adjoint operators. Although the defining formula for Lusztig's isomorphism:
\Eq{T_i(E_i):=-F_iK_i\inv,\tab T_i(E_j):=q^\half E_jE_i-q^{-\half} E_iE_j}
in the compact case does not respect the positivity of the positive representations, it turns out that we can consider instead our favorite rescaled generators $\{\be_i, \bf_i\}$ from \eqref{smallef}, and modify the scaling to obtain
\Eq{T_i(\be_i):=q \bf_iK_i \inv,\tab T_j(\be_i):=\be_{ij}:=\frac{q^\half \be_j\be_i-q^{-\half}\be_i\be_j}{q-q\inv},}
which are again positive self-adjoint, and obey the transcendental relations 
\Eq{\be_{ij}^{\frac{1}{b^2}}=\til{\be_{ij}}:=\frac{\til{q}^\half \til{\be_j}\til{\be_i}-\til{q}^{-\half}\til{\be_i}\til{\be_j}}{\til{q}-\til{q}\inv}.}
Furthermore, $T_i$ is shown to be represented by conjugation by a unitary element $w_i$, called the quantum Weyl element. Hence all generators of the form (cf. \eqref{nonsimple})
\Eq{\label{eak}\be_{\a_k}:=T_{i_1}...T_{i_{k-1}}(\be_{i_k})}
are positive self-adjoint and satisfy the transcendental relations. In particular, the expression $\be_{\a_k}^{ib\inv t}$ is a well-defined bounded unitary operator acting on $\cP_\l$, and invariant under the $b\corr b\inv$ duality.

Therefore together with the notation $K_i=q^{H_i}$, we generalize the PBW-basis \eqref{cbasis} and the following definition is well-defined \cite{Ip4}:
\begin{Def} \label{Ub}Let $n=rank(\g)$ and $N=l(w_0)$. We define the $C^*$-algebraic version of the Borel subalgebra $$\bU\bb:=\cU_{q\til{q}}^{C^*}(\fb_\R^+)$$ as the operator norm closure of the linear span of all bounded operators on $\cP_\l=L^2(\R^N)$ of the form
\Eq{\vec{F}:= F_0(\bH)\prod_{k=1}^N  \int_C F_k(t_k)\frac{\be_{\a_k}^{i b\inv t_k}}{G_{b}(Q+i t_k)} dt_k,}
where $Q=b+b\inv$, $\be_{\a_k}$ is given by \eqref{eak} and
\Eq{F_0(\bH):=F_0(i bH_1,...,i bH_n)}
is a smooth compactly supported functions on the positive operators $i b H_k$, and
$F_k(t_k)$ are entire analytic functions that have rapid decay along the real direction. Finally the contour $C$ is along the real axis which goes above the pole of $G_b$ at $t_k=0$.
\end{Def}
\begin{Thm} \begin{itemize}
\item[(1)] $\bU\bb$ is a multiplier Hopf algebra;
\item[(2)] it is invariant under the modular duality $b\corr b\inv$;
\item[(3)] it is independent of the choice of the parameter $\l$;
\item[(4)] it is independent of the choice of reduced expression of the longest element $w_0\in W$.
\end{itemize}
\end{Thm}
\begin{proof}
(1) The quantum dilogarithm function $G_b$ appearing in the definition ensures that $\bU\bb$ is indeed a multiplier Hopf algebra, where the coproduct of $\be_{\a_k}$ will pick up factors of $G_b(Q+it_k)$ due to the $q$-binomial formula (Lemma \ref{qbinom}) and mutually cancel, giving the correct analytic properties of the bounded operators. We see immediately that $G_b(Q+it)$ plays the role of the $q$-factorial $[n]_q!$ factors in the compact canonical basis. 

(2) The transcendental relations \eqref{trans} and the duality properties \eqref{selfdual} of $G_b$ ensures that the expression is invariant under $b\corr b\inv$, hence it indeed incorporates the properties of the modular double. 

(3) From the explicit expression of $\cP_\l$, one notes that the dependence of $\l$ only appears as linear terms in $H_i$, and the positive representation is injective (see \cite{Ip2}). Hence the algebra does not depend on $\l$. In fact, there exists a unitary equivalence between $\cP_\l\simeq \cP_{\l'}$ for any $\l,\l'$ when restricted to the Borel part, considered in \cite{Ip5}.

(4) Finally, we prove in \cite{Ip4} that the continuous span by the basis $\prod_{k=1}^N \frac{\be_{\a_k}^{i b\inv t_k}}{G_b(Q+it_k)}$ does not depend on the choice of the expression of the longest element $w_0$. More precisely, the change of words of $w_0$ induces an integral transformation on the coefficients $F_k(t_k)$ that preserves their analytic properties. (In the non-simply-laced case however, the explicit formula is not known for type $G_2$.) It also shows the existence of the antipode $S$ on the multiplier Hopf algebraic level as an unbounded operator.
\end{proof}

Now we can apply the construction of Drinfeld's double in the setting of multiplier Hopf algebra \cite{DvD} to obtain $\cD:=\cD(\bU\bb)$, which is known \cite{DvD2} to be a quasi-triangular multiplier Hopf algebra, where $R$ is given by the canonical element, which is the unique element in $M(\cD\ox \cD)$ such that 
\Eq{\<R,b\ox a\>=\<a,b\>,\tab a\in\cA, b\in \cA'.}
\begin{Def} We define 
\Eq{\bU:=\cU_{q\til{q}}^{C^*}(\g_\R):=\cD(\bU\bb)/(H_i'=(A\inv)_{ij}H_j)}
 to be the Drinfeld's double of the Borel subalgebra $\bU\bb$ modulo the Cartan subalgebra $\fh\subset \bU\bb^-$.
\end{Def}
It follows naturally from the explicit expression of the universal $R$-operator constructed in \cite{Ip4} which generalizes \cite{KR,LS}:
\Eq{R=&\prod_{ij}q^{\half(A\inv)_{ij}H_i\ox H_j}\prod_{k=1}^{N} g_b(\be_{\a_k}\ox \bf_{\a_k}) \prod_{ij}q^{\half(A\inv)_{ij}H_i\ox H_j},}
together with the integral transformation \eqref{Gg} of $g_b$ (replacing the usual $q$-exponential function), and the Hopf pairing between $\bU\bb$ and $\bU\bb^-$ involving $G_b(x)$ (replacing the usual $q$-factorial), that the canonical element is given precisely by $R$. Therefore $\bU$ is indeed a quasi-triangular multiplier Hopf algebra. Furthermore, due to Lemma \ref{unitary}, $R$ induces a unitary transformation between $\cP_{\l_1}\ox \cP_{\l_2}\simeq \cP_{\l_2}\ox \cP_{\l_1}$, hence giving the braiding structure of the positive representations.

\begin{Rem}
In the compact case, Lusztig's canonical basis is certain linear combinations of the PBW-basis that is invariant under the bar involution $q\corr q\inv$. In the non-compact case, it is expected that the true notion of a continuous canonical basis will be certain integral transformations of the continuous PBW-basis constructed above, which obeys some form of non-compact bar involutions involving complex conjugations. Here the positivity condition is expected to play an important role in its definition.
\end{Rem}

\section{Tensor Products of $\cP_\l$}\label{sec:tensor}

The notion of continuous canonical basis may provide the key to study the decomposition of tensor products of positive representations. In the finite dimensional situation generated by a highest weight vector $\xi_\l$ with dominant integral weight $\l$, the $\cU_q(\g)$ module is generated by basis of the form $\{E_\bi^\bc \xi_\l\}$ where $E_\bi^\bc$ is given as in \eqref{cbasis}. The decomposition of tensor product is then, roughly speaking, given by the tensor product of the corresponding basis giving the same weight, modulo certain monomial terms with higher order of $q$ in the coefficients \cite{Lu3}. In the setting of positive representations, there are no highest weight vector anymore, but still one can fix an initial weight vector $\xi_\l$ and define $\cP_\l$ to be ``generated" by generalized functions such as
\Eq{\cP_\l=\mbox{``continuous span"}\left\{\prod_{k=1}^N\frac{\be_{\a_k}^{i b\inv t_k}}{G_{b}(Q+i t_k)}\xi_\l\right\}.}
As an example, in the case of $\cU_{q\til{q}}(\sl(2,\R))$, one can regard $\cP_\l=L^2(\R)$ as generated by the initial weight vector $\d(x)$, (where $K$ acts as $e^{-2\pi b\l}$), and look at the action of $\be^{ib\inv t}$ on $\d(x)$ \cite{Ip4}: 
\Eq{\be^{i b\inv t}\cdot \d(x)=&e^{\frac{\pi i(t-2\l)t}{2}}\frac{G_b(\frac{Q}{2}-i\l+it)}{G_b(\frac{Q}{2}-i\l)}\d(x-t).}
In particular, in the analytic continuation $t\to -ib$, we recover the action of $\be$ given in Proposition \ref{canonicalsl2}.

One can then take these continuous basis and follow the decomposition in the compact case, generalizing the formula for the decomposition $\cP_{\l_1}\ox \cP_{\l_2}$. In principle this will recover the intertwiner:
\Eq{\cP_{\l_1}\ox \cP_{\l_2}\simeq \int_{\R_+}^{\o+} \cP_{\l} \sinh(2\pi b\l)\sinh(2\pi b\inv \l)d\l}
given by a series of integral transformations involving $G_b(x)$ proved in \cite{NT, PT2}.

Of course, one has to be careful when dealing with such generalized functions in terms of functional analysis. Nevertheless, investigating this observation, we have the following more precise conjecture for the higher rank case.
\begin{Conj} The positive representations for $\cU_{q\til{q}}(\g_\R)$ is closed under taking the tensor product, and it decomposes as
\Eq{\cP_{\a}\ox \cP_{\b}\simeq \int_{\R_+^N}^{\o+} \cP_{\vec{\c}} d\mu(\vec{\c}),}
where $\vec{\c}=\sum_{\a\in \D_+} \c_\a\w_\a$ suming over all the positive roots, where $\c_\a\in \R_+$ and $\w_\a$ are the fundamental weights, with the abuse of notation $\w_\a:=\w_{\a_1}+\w_{\a_2}$ if $\a:=\a_1+\a_2$ is not simple. The Plancherel measure is given by $$d\mu(\vec{\c})=\prod_{\a\in\D_+}\sinh(2\pi b\c_\a)\sinh(2\pi b\inv \c_\a)d\c_\a.$$
\end{Conj}
\begin{Rem}
It was pointed out by Masahito Yamazaki that this Plancherel measure naturally comes from a somewhat indirect connection from supersymmetric gauge theory, which is discussed for example in \cite{DGG,HLP}.
\end{Rem}

In particular, in the simply-laced case, we believe it suffices to prove
\begin{Conj}The tensor product of positive representations for $\cU_{q\til{q}}(\sl(3,\R))$ decomposes as
\Eq{\cP_{(\a_1,\a_2)}\ox \cP_{(\b_1,\b_2)}\simeq \int_{\R_+^3}^{\o+} \cP_{(\c_1+\c_3,\c_2+\c_3)} d\mu(\vec{\c}).}
\end{Conj}

Finally, let us remark that the situation is simpler if we restrict the positive representations $\cP_\l$ to the Borel part $\fb_\R\subset \g_\R$. Using the notion of $\cU_{q\til{q}}^{C^*}(\fb_\R)$ and the theory of multiplicative unitary associated to its GNS-representations, it is shown in a recent study \cite{Ip5} that the positive representations $\cP_\l$ restricted to $\fb_\R$ does not depend on $\l$, and it is closed under taking tensor products. This construction induces the quantum mutation operator $\bT$ which will be a key ingredient to obtain a candidate for the quantum higher Teichm\"{u}ller theory, generalizing earlier work by Frenkel-Kim \cite{FK}.

\section{Harmonic Analysis of $L^2(G_{q\til{q}}^+(\R))$}\label{sec:harmonic}
In the last section, we would like to discuss the application of multiplier Hopf algebra in the context of harmonic analysis on the quantized algebra of functions on $G_{q\til{q}}^+(\R)$, which is certain dual algebra to $\cU_{q\til{q}}(\g_\R)$. This is motivated from the following form of Peter-Weyl theorem \cite{Ip1, PT1}
\Eq{L^2(SL_{q\til{q}}^+(2,\R))\simeq \int_{\R_+}^\o+ \cP_\l\ox \cP_\l^* d\mu(\l)}
as left and right regular representations of $\cU_{q\til{q}}(\sl(2,\R))$, where the Plancherel measure $d\mu(\l)$ is the same as before.

Here, the Hilbert space $L^2(SL_{q\til{q}}^+(2,\R))$ is constructed as follows. Starting with the \emph{quantum plane} generated by positive self-adjoint elements $\{A,B\}$ with ${AB=q^2BA}$, one define the $C^*$-algebraic version $\cA_{q\til{q}}$ similar to the construction of $\bU\bb$, giving it a multiplier Hopf algebra structure, as well as imposing the modular duality $b\corr b\inv$. With the existence of a Haar functional and certain density conditions, $\cA_{q\til{q}}$ can be shown to be a locally compact quantum group in the sense of \cite{KV1,KV2}. Finally, by applying Woronowicz's \emph{quantum double construction}, which is the dual of Drinfeld's double construction, we obtain a multiplier Hopf algebra $\cA(SL_{q\til{q}}^+(2,\R))$, and the Hilbert space $L^2(SL_{q\til{q}}^+(2,\R))$ is defined by its von Neumann completion using the GNS representation. The analysis utilizes the power of the multiplicative unitary $W$ associated to the GNS representation of the quantum plane, and the corresponding quantum double counterparts, which is the key ingredient to construct the regular representations of $\cU_{q\til{q}}(\sl(2,\R))$.

One natural question is the extension of this result to the higher rank. In \cite{Ip5}, we define the \emph{Gauss-Lusztig's decomposition} (for type $A_n$), which plays the role of Woronowicz's quantum double construction. Here using Lusztig's data for the totally positive matrices, the space $G_{q\til{q}}^+(\R)$ is given by an algebra generated by elements that form interconnecting $q$-tori. Different parametrization of the Lusztig's data gives intertwiners that is also closely related to the positive representations on a combinatorial level \cite{Ip2, KOY}. Imposing positivity, we can mimic the construction in the case of the quantum plane and give this algebra a multiplier Hopf algebra structure, as well as it's von Neumann completion, hence the Hilbert space $L^2(G_{q\til{q}}^+(\R))$. One can then discuss the regular representations of $\cU_{q\til{q}}(\g_\R)$ acting on this space, and give the following conjecture.
\begin{Conj} We have the decomposition
\Eq{L^2(G_{q\til{q}}^+(\R))\simeq \int_{\R_+^n}^\o+ \cP_\l\ox \cP_\l^* d\mu(\l)}
as left and right regular representations of $\cU_{q\til{q}}(\g_\R)$. This time, however, $\l$ runs through the fundamental weights $\w_i$ corresponding to simple roots only.
\end{Conj}
\begin{appendices}
\section{Quantum Dilogarithms}
First introduced by Faddeev and Kashaev \cite{FKa}, the quantum dilogarithm $G_b(x)$ and its variant $g_b(x)$ play a crucial role in the study of positive representations of split real quantum groups, and also appear in many other areas of mathematics and physics. In this appendix, let us recall the definition and some properties of the quantum dilogarithm functions \cite{BT, Ip1, PT2} that is needed in this paper.

\begin{Def} The quantum dilogarithm function $G_b(x)$ is defined on\\ ${0\leq Re(z)\leq Q}$ by
\Eq{\label{intform} G_b(x)=\over{\ze_b}\exp\left(-\int_{\W}\frac{e^{\pi tz}}{(e^{\pi bt}-1)(e^{\pi b\inv t}-1)}\frac{dt}{t}\right),}
where \Eq{\ze_b=e^{\frac{\pi \bi }{2}(\frac{b^2+b^{-2}}{6}+\frac{1}{2})},}
and the contour goes along $\R$ with a small semicircle going above the pole at $t=0$. This can be extended meromorphically to the whole complex plane with poles at $x=-nb-mb\inv$ and zeros at $x=Q+nb+mb\inv$, for $n,m\in\Z_{\geq0}$;
\end{Def}
The quantum dilogarithm $G_b(x)$ satisfies the following properties:
\begin{Prop} Self-duality:
\Eq{G_b(x)=G_{b\inv}(x);\label{selfdual}}

Functional equations: \Eq{\label{funceq}G_b(x+b^{\pm 1})=(1-e^{2\pi \bi b^{\pm 1}x})G_b(x);}

Reflection property:
\Eq{\label{reflection}G_b(x)G_b(Q-x)=e^{\pi \bi x(x-Q)};}

Complex conjugation: \Eq{\overline{G_b(x)}=\frac{1}{G_b(Q-\bar{x})},\label{Gbcomplex}}
in particular \Eq{\label{gb1}\left|G_b(\frac{Q}{2}+\bi x)\right|=1 \mbox{ for $x\in\R$};}

\label{asymp} Asymptotic properties:
\Eq{G_b(x)\sim\left\{\begin{array}{cc}\bar{\ze_b}&Im(x)\to+\oo\\\ze_b
e^{\pi \bi x(x-Q)}&Im(x)\to-\oo\end{array},\right.}
\end{Prop}
\begin{Lem}[$q$-binomial theorem]\label{qbinom}
For positive self-adjoint variables $U,V$ with $UV=q^2VU$, we have:
\Eq{(V+U)^{i b\inv t}=\int_{C}\veca{i t\\ i \t}^b V^{i b\inv (t-\t)}U^{i b\inv \t}d\t ,}
where the $q$-beta function (or $q$-binomial coefficient) is given by
\Eq{\veca{t\\\t}^b=\frac{G_b(Q+t)}{G_b(Q+\t)G_b(Q+t-\t)},}
and $C$ is the contour along $\R$ that goes above the pole at $\t=0$
and below the pole at $\t=t$.
\end{Lem}
\begin{Def} The function $g_b(x)$ is defined by
\Eq{g_b(x)=\frac{\over{\ze_b}}{G_b(\frac{Q}{2}+\frac{\log x}{2\pi \bi b})},}
where $\log$ takes the principal branch of $x$.
\end{Def}
\begin{Lem}\label{FT} \cite[(3.31), (3.32)]{BT} We have the following Fourier transformation formula:
\Eq{\int_{\R+i 0} \frac{e^{-\pi i  t^2}}{G_b(Q+\bi t)}X^{i b\inv t}dt=g_b(X) \label{Gg},}
where $X$ is a positive operator and the contour goes above the pole
at $t=0$.
\end{Lem}

\begin{Lem}\label{unitary} By \eqref{gb1}, $|g_b(x)|=1$ when $x\in\R_+$, hence $g_b(X)$ is a unitary operator for any positive operator $X$. Furthermore we have the self-duality of $g_b(x)$ given by
\Eq{\label{selfdualg}g_b(X)=g_{b\inv}(X^{\frac{1}{b^2}}).}
\end{Lem}
\end{appendices}
\section*{Acknowledgments} I would like to thank the organizing committee of the RIMS symposium on String Theory, Integrable Systems and Representation Theory for their invitation. I would also like to thank in particular Hiraku Nakajima, Masato Okado, Tomoki Nakanishi and Yoshiyuki Kimura for stimulating interesting discussions and providing useful comments. This work was supported by World Premier International Research Center Initiative (WPI Initiative), MEXT, Japan.

\end{document}